\documentclass[12pt,reqno]{article}

\usepackage{graphicx}
\usepackage{amscd}
\usepackage{color}
\usepackage{graphics,amsmath,amssymb}
\usepackage{amsthm}
\usepackage{amsfonts}

\usepackage{float}
\usepackage{tikz}
\usepackage{forest}
\usepackage{subfigure}

\usepackage{verbatim}
\usepackage{enumerate}
\usepackage{hyperref}
\usepackage{listings}

\usepackage[blocks]{authblk}

\usepackage{multicol}
\usepackage{indentfirst}
\usepackage{changepage}
\usepackage{enumitem,kantlipsum}
\usepackage{thm-restate}
\usepackage{breqn}

\usepackage{tikz}
\usepackage{tikz-qtree}

\theoremstyle{plain}
\newtheorem{theorem}{Theorem}

\newtheorem{corollary}[theorem]{Corollary}
\newtheorem{lemma}[theorem]{Lemma}

\theoremstyle{definition}
\newtheorem{definition}{Definition}
\newtheorem{example}{Example}

\theoremstyle{remark}
\newtheorem*{remark}{Remark}

\newcounter{row}
\newcounter{col}

\newcommand\setrow[4]{
  \setcounter{col}{1}
  \foreach \n in {#1, #2, #3, #4} {
    \edef\x{\value{col} - 0.5}
    \edef\y{4.5 - \value{row}}
    \node[anchor=center] at (\x, \y) {\n};
    \stepcounter{col}
  }
  \stepcounter{row}
}

\newlength{\stretchlen}
\def\splitterm{\_}
\newcommand{\stretchit}[1]{\leavevmode\realstretch#1\_}
\def\realstretch#1{%
    \def\temp{#1}%
    \ifx\temp\splitterm
    \else
    \hbox to \stretchlen{\hss#1\hss}\expandafter\realstretch
\fi}

\begin{document}

\title{Fractal Word Search: How Deep to Delve}
\author[1]{Klára Churá}
\author[2]{Tanya Khovanova}
\affil[1]{Tufts University}
\affil[2]{Massachusetts Institute of Technology}
\maketitle

\section{Abstract}

We look at the puzzle \textit{In the Details} which appeared in the 2013 MIT Mystery Hunt and which gained fame as the \textit{fractal word search}. This seemingly impossible puzzle, whose solution could not fit the memory of a modern computer if the puzzle were solved using a brute-force approach, requires an understanding of its fundamental structure to be cracked. In this paper, we study fractal word searches in a general setting, where we consider one- and two-dimensional word searches with alphabets of any length and replacement rules of any size. We prove that the puzzle is solvable within a finite number of steps under this generalization and give an explicit upper bound on the latest level on which a word of a given length can appear for the first time in a given direction.

\section{Introduction}
\label{section:introduction}

\subsection{Background}

In the 2013 MIT Mystery Hunt, a puzzle titled \textit{In the Details} appeared, authored by Derek Kisman. The puzzle, which can also be found at \cite{kisman2013puzzle}, is written out in Figure~\ref{fig:puzzle}.

\begin{figure}
\begin{center}
\stretchit{TWELEVELTWONSHELMUMUOERAIYRANL}\par
\stretchit{QAPIUNPIQAYDPEPIRPRPKVOYESOYOR}\par
\stretchit{ELRATFDTELDTTFDTBWNLMUTFONYDWJ}\par
\stretchit{PIOYJMHAPIHAJMHAAOORRPJMYDANFC}\par
\stretchit{MUOZCGTFBWIRYDHIRAIRTFNCUENCUE}\par
\stretchit{RPVQUHJMAOHKANJUOYHKJMZKBNZKBN}\par
\stretchit{IRONSHOZGOTFUEELTFOEELUEYDOETF}\par
\stretchit{HKYDPEVQDNJMBNPIJMKVPIBNANKVJM}\par
\stretchit{BWIYNLTFSHHIELTWGOYDONDTYDHIOE}\par
\stretchit{AOESORJMPEJUPIQADNANYDHAANJUKV}\par
\stretchit{SHDTYDRPBWUEBWIYTWTWTFYDMUELMU}\par
\stretchit{PEHAANAJAOBNAOESQAQAJMANRPPIRP}\par
\stretchit{ONTWELBWLMSHELTFUEBWBWLMOZEVHI}\par
\stretchit{YDQAPIAOGIPEPIJMBNAOAOGIVQUNJU}\par
\stretchit{DTCGUEYDRPEVNCIREVIRTWUEUETWON}\par
\stretchit{HAUHBNANAJUNZKHKUNHKQABNBNQAYD}\par
\stretchit{IRUERAMUTFELTWONTFOEOEEYDTNLYD}\par
\stretchit{HKBNOYRPJMPIQAYDJMKVKVHWHAORAN}\par
\stretchit{ELGORPNCTFDTYDSHYDELPKTFOZRACG}\par
\stretchit{PIDNAJZKJMHAANPEANPIDFJMVQOYUH}\par
\stretchit{DTMUWJOETFYDELMUMUGORAONIRDTCG}\par
\stretchit{HARPFCKVJMANPIRPRPDNOYYDHKHAUH}\par
\end{center}

\begin{changemargin}{-2em}{-2em}
\begin{multicols}{4}
\begin{description}
\itemsep-0.5em
    \item BOUNDARY
    \item BROWNIAN
    \item CAUCHY
    \item CURLICUE
    \item DE RHAM
    \item DIMENSION
    \item ESCAPE
    \item HAUSDORFF
    \item HENON
    \item HILBERT
    \item HURRICANE
    \item ITERATE
    \item JULIA
    \item LEIBNIZ
    \item LEVEL ONE
    \item LEVEL TWO
    \item LEVY DRAGON
    \item LYAPUNOV
    \item MANDELBROT
    \item NEURON
    \item NURNIE
    \item POWER LAW
    \item RAUZY
    \item RIVER
    \item SCALING
    \item SPACE
    \item STRANGE
    \item TAKAGI
    \item TECTONICS
    \item T-SQUARE
    \item WIENER
    \item YO DAWG
\end{description}
\end{multicols} 
\end{changemargin}

\begin{center}
\begin{tabular}{ c c c c c c c c }
\textunderscore & \textunderscore & \textunderscore & \textunderscore & \textunderscore & \textunderscore & \textunderscore & \textunderscore \\
\end{tabular}
\end{center}
\caption{In the Details}
\label{fig:puzzle}
\end{figure}

This puzzle has since gained the name \textit{fractal word search}. Let us have a look at it together to discover why!

First, we notice that the puzzle looks like a word search. In a word search, one can find words from a given list in a given grid looking horizontally, vertically, or diagonally, in eight directions overall. After all the words have been found and their letters crossed out, the leftover letters in the grid spell out the answer or at least a hint on it. Traditionally, the number of blanks in the final row of the puzzle represents the length of the sought answer, which in the case of our puzzle is a length of $8$.

However, this puzzle is not a regular word search since only $6$ words from the list can be found in the grid. That means that the puzzle has a hidden secret, but what is this secret? We can observe that the words in the list are related to fractals; this is the first big hint. The second hint is that the grid shows many repeating $2 \times 2$ blocks, of which there are exactly $26$ kinds; this suggests that each block represents a letter in the alphabet and that each can be replaced with this letter, yielding a grid smaller by a factor of two in each dimension. Finally, to make things a little easier, there are the words $LEVEL TWO$ in the upper left corner of the original grid. Could there be the words $LEVEL ONE$ on the first line of the smaller grid we are looking for? This happens to be the case, and the matching between letters and two-by-two blocks is as in Figure~\ref{fig:replacementrules}.

\begin{figure}
\begin{changemargin}{-2em}{-2em}
\begin{center}
\begin{tabular}{ l l l l l l l l l l l l l }
A & B & C & D & E & F & G & H & I & J & K & L & M \\

TF & GO & IR & HI & EL & CG & RP & UE & BW & PK & OZ & TW & NL \\

JM & DN & HK & JU & PI & UH & AJ & BN & AO & DF & VQ & QA & OR \\ [1em]

N & O & P & Q & R & S & T & U & V & W & X & Y & Z \\

SH & ON & RA & WJ & YD & MU & DT & OE & EV & NC & EY & IY & LM \\

PE & YD & OY & FC & AN & RP & HA & KV & UN & ZK & HW & ES & GI \\
\end{tabular}
\end{center}
\end{changemargin}
\caption{Matching between letters and two-by-two blocks}
\label{fig:replacementrules}
\end{figure}

That is, level one is represented by the $11 \times 15$ grid in Figure~\ref{fig:levelone}.

\begin{figure}
\def\realstretch#1{%
    \def\temp{#1}%
    \ifx\temp\splitterm
    \else
    \hbox to \stretchlen{\hss#1\hss}\expandafter\realstretch
\fi}
\begin{center}
\stretchit{LEVELONESSUPYPM}\par
\stretchit{EPATETATIMSAORQ}\par
\stretchit{SKFAICRDPCAWHWH}\par
\stretchit{CONKBAHEAUEHRUA}\par
\stretchit{IYMANDELBROTRDU}\par
\stretchit{NTRGIHIYLLARSES}\par
\stretchit{OLEIZNEAHIIZKVD}\par
\stretchit{TFHRGVWCVCLHHLO}\par
\stretchit{CHPSAELOAUUXTMR}\par
\stretchit{EBGWATRNREJAKPF}\par
\stretchit{TSQUARESSBPOCTF}\par
\end{center}
\caption{Level one}
\label{fig:levelone}
\end{figure}

This level contains $18$ more words from the list, but multiple words are still missing. Hence, we are motivated to continue to level three and perhaps beyond in a search for the rest. Going back to level two and replacing its letters with blocks, we find an additional three words from the list on level three and one word on level four. However, even level four is already a bit too big for inspection, and since the subsequent levels grow in size exponentially, there is no way we could solve the puzzle in reasonable time using the brute-force approach of repeatedly replacing all letters and looking for words on every vertical, horizontal, and diagonal.

Thus, we need to step back and look at how the puzzle works so that we can come up with a smarter approach. How do fractal puzzles work? How deep can a word be hidden that does not appear on any previous level, and is this affected by the characteristics of the original grid, such as the size of the alphabet, the lengths of the sides of the replacement rules, or the length of the word we are looking for? Can the search for a word hidden very deep be done in a more optimal way than by writing out all the levels? These questions will be the very focus of our paper.

\subsection{Overview}

This subsection covers our main results. For precise definitions of any of the notation below, see the preliminaries in Section~\ref{section:preliminaries}.

First, in Section~\ref{section:onedimension}, we gain an intuition behind Kisman's puzzle by considering its equivalent in one dimension, with an $n$-letter alphabet and $b$-letter replacement rules. We then define a function $\mathcal{W}_1$ on $n$, $b$, and the length $|w|$ of a given word $w$, and we prove in Theorem~\ref{thm:onedim} that $\mathcal{W}_1$ is an upper bound on $F_{1}(b,n,w)$, where $F_{1}(b,n,w)$ is a hypothetical function giving the actual latest level on which a specific word $w$ can first appear across all possible setups of level one, denoted $L_1$.

Having focused on the puzzle in one dimension, we use Section~\ref{section:twodimensions} to extend our understanding of the fractal word search to two dimensions. We prove in Theorem~\ref{thm:twodimvert} that the function $\mathcal{W}_1$ defined in Section~\ref{section:onedimension} can also be used as an upper bound for $F_{2}^{hv}(b,n,w)$, a function that provides the latest level on which a horizontal or vertical word $w$ can appear for the first time. Again, $n$ is the size of the alphabet, $b$ is the length of the side of the square replacement rules, and $w$ is the word in question.

The result we obtain in Theorem~\ref{thm:twodimvert}, however, does not extend to diagonal words. Hence, we define another function $\mathcal{W}_2$, which takes as its input the length of the alphabet $n$, the side-length of the square replacement rules $b$, and the length $|w|$ of the sought word $w$. We prove in Theorem~\ref{thm:twodimdiag} that this $\mathcal{W}_2$ is an upper bound on $F_{2}^{d}(b,n,w)$, where $F_{2}^{d}(b,n,w)$ gives the actual latest level on which a specific diagonal word $w$ can appear for the first time across all possible setups of $L_1$, given an $n$-letter alphabet and $\left(b \times b\right)$-letter replacement rules.

The rest of the paper is organized in the following way. In Section~\ref{section:preliminaries}, we codify our notation and provide any preliminary definitions. In Section~\ref{section:onedimension}, we prove Theorem~\ref{thm:onedim}. In Section~\ref{section:twodimensions}, we prove Theorems~\ref{thm:twodimvert} and \ref{thm:twodimdiag}. Finally, in Section~\ref{section:solution}, we use our results to solve the puzzle, and in Section~\ref{section:conclusion}, we evaluate our findings and discuss potential future topics of exploration.

\section{Preliminaries}
\label{section:preliminaries}

We consider an alphabet on $n$ letters and use upper-case Latin letters for the letters in the alphabet, e.g. $A$, $B$, or $C$. We then call any string of letters from the alphabet a \textit{word}. For instance, $CAT$ and $ARRRRGGG$ are both words constructed from the letters of an alphabet that contains at least the letters $A$, $C$, $G$, $R$, and $T$. We use an asterisk $*$ to identify an unknown letter in a word of a known length, e.g., $AT*$ could stand for $ATA$ or $ATT$ or any other $3$-letter word beginning with $AT$. Given a word $w$, we denote its length by $|w|$. We also refer to a generalized rectangle of letters as a \textit{block}.

In Kisman's word search, as seen in \cite{kisman2013solution}, words in most of the eight possible directions are permissible. However, to make the discussion easier, we will focus only on horizontal words from left to right, vertical words from top to bottom, and diagonal words from the upper-left to the lower right, and we will thus use the simplified terms \textit{horizontal}, \textit{vertical}, and \textit{diagonal} to refer to these specific directions. Note that the remaining five directions will be analogous to those just chosen.

In our alphabet, each letter has a \textit{replacement rule} of a given size; in this paper, we restrict ourselves to the study of one-dimensional \textit{grids}, where the letters have one-dimensional $b$-letter replacement rules, and two-dimensional grids, where individual letters have square replacement rules of $\left(b \times b\right)$ letters. Starting with a grid of letters, we can replace each letter in this grid according to its replacement rules and thus obtain a new, enlarged grid. We can formalize this notion by defining a projection $R(p)$, which maps a block of letters $p$ to its replacement based on the replacement rules of the individual letters of $p$. We refer to the starting grid as the $1$-st \textit{level} or $L_1$ and denote any consecutive $k$-th level as $L_k$, and $R(p)$ can be seen as a mapping between the individual levels. For example, $R(L_k) = L_{k+1}$, for $k \geq 1$.

\begin{example}
\label{example:alphabetonedim}
Take the $3$-letter alphabet consisting of the letters $A$, $B$, and $C$, where the replacement rules of these letters are $AB$, $AC$, and $BB$ respectively. If $L_1$ has the single letter $A$, we get the consecutive levels $L_2 = R(L_1) = R(A) = AB$, $L_3 = R(L_2) = R(AB)=ABAC$, and $L_4 = R(L_3) = R(ABAC)=ABACABBB$. Thus, we can catch a $CAB$ on $L_4$. The reader can also notice that with this setup, the word $CC$ can never appear.
\end{example}

We refer to a word $v$ on $L_{k-1}$ as the \textit{parent} of the word $w$ if $R(v)$ contains $w$ and $R(v')$ does not contain $w$ for any sub-word $v'$ of $v$. In this case, we call $w$ a \textit{child} of $v$.

\begin{example}
The parent of $CAB$ on $L_4$ in Example~\ref{example:alphabetonedim} is $BA$ on $L_3$.
\end{example}

Note that $R(p)$, defined above, does not necessarily have a well-defined inverse. Although a block of letters located in a particular place of the grid on a particular level has a unique parent, the same word in a different place might have a different parent. 

\begin{example}
Using the alphabet and replacement rules of Example~\ref{example:alphabetonedim}, $BA$ could be a child of $AA$, $AB$, $CA$, or $CB$.
\end{example}

The main interest of this article lies in determining the latest level on which a given word $w$ can appear for the first time across all the possible setups of $L_1$ if it ever appears. Formally, we can view this as a goal function $F(w)$ which takes a word as the input and outputs the corresponding latest level.

\begin{example}
Given the alphabet in Example~\ref{example:alphabetonedim}, we can examine when the word $A$ appears with different setups of $L_1$. If $L_1$ contains $A$, then $A$ appears for the first time there. If $L_1$ does not contain $A$ but contains $B$, then $A$ appears for the first time on $L_2$. If $L_1$ has only instances of the letter $C$, then $A$ appears for the first time on $L_3$. Thus, $F(A)=3$.
\end{example}

However, the above definition of $F(w)$ is too vague for the scope of this study, for it does not take into account the length of the alphabet $n$, the size of the replacement rules, the dimensionality of the grid, or the direction of the word when the grid is multi-dimensional. Hence, we make the function explicit by introducing $F_{1}(b,n,w)$ for the search of words $w$ in a one-dimensional grid where letters have $b$-letter replacement rules, as well as $F_{2}^{hv}(b,n,w)$ and $F_{2}^{d}(b,n,w)$ for the search of words $w$, horizontal/vertical and diagonal respectively, in a two-dimensional grid where letters have $\left(b \times b\right)$-letter replacement rules.

\section{One Dimension}
\label{section:onedimension}

In this section, we prove an upper bound on the latest level on which a word can appear for the first time in one dimension, with an $n$-letter alphabet and $b$-letter replacement rules. Many of the results given here are directly relevant to the exploration of vertical, horizontal, and diagonal words in the two-dimensional scenario.

We first make an observation that will help us better understand the transition between individual levels. This observation is the first step towards concluding that any fractal word search is finite.

\begin{lemma}
\label{lemma:firsttimelk}
If a word $w$ appears for the first time on $L_{k}$, then any parent of $w$ appears for the first time no earlier than on $L_{k-1}$. Moreover, at least one of the parents appears on $L_{k-1}$.
\end{lemma}

\begin{proof}
By contradiction for the first part of the statement, assume a parent $p$ of $w$ appears on $L_{j}$, where $j<k-1$. Then, since $w$ is a child of $p$, $w$ appears on $L_{j+1}$. However, since $j+1<k$, this contradicts our initial assumption that $w$ appears for the first time on $L_{k}$. Thus, the first time a parent of $w$ can appear is no earlier than $L_{k-1}$. For the second part of the statement, if none of the possible parents of $w$ appears on $L_{k-1}$, then the word $w$ itself does not appear on $L_{k}$.
\end{proof}

With this observation in hand, we can establish the latest level on which a $1$-letter word can appear for the first time, whatever the size of the alphabet and even regardless of the dimensionality of the grid and the size of the replacement rules. The following result serves as the base case for the latest first-time appearance of a word of any length.

\begin{lemma}
\label{lemma:oneletterword}
Given an $n$-letter alphabet and a $1$-letter word $w$, the latest level on which $w$ can appear for the first time is $L_{n}$:
\[F_{1}(b,n,w) \leq n.\]
\end{lemma}

\begin{proof}
The parent of a single letter is always itself a single letter. When tracing back the single-letter parents of parents of $w$, we see, using Lemma~\ref{lemma:firsttimelk}, that all of them must be unique. Hence, we can only go back $n-1$ levels before we run out of letters that appear for the first time on any given level. Thus, $w$ appears for the first time on $L_{n}$ at the latest, else it does not appear at all.
\end{proof}

\begin{example}
In Example~\ref{example:alphabetonedim}, the size of the alphabet is $3$. Letter $A$ appears on $L_1$ for the first time, $B$ on $L_2$, and $C$ on $L_3$. No new letter appears from $L_4$ onwards since there are no other letters left in the alphabet. Thus, all new $1$-letter words appear for the first time within the bound $3$ given by Lemma~\ref{lemma:oneletterword}.
\end{example}

The case of $1$-letter words above is, in fact, special, for $1$-letter words are the only words that can never, on any level, appear on the overlap of two or more replacement rules of letters from the previous level. Thus, $2$-letter words represent another important base case, namely that of words that can span multiple blocks given by the replacement rules. Thus, understanding $2$-letter words in the one-dimensional scenario does most of the job of understanding words of any length.

\begin{lemma}
\label{lemma:twoletterwordonedim}
Given an $n$-letter alphabet and a $2$-letter word $w$, the latest level on which $w$ can appear for the first time is no greater than $L_{n^2+1}$:
\[F_{1}(b,n,w) \leq n^2+1.\]
\end{lemma}

\begin{proof}
Suppose $w$ appears for the first time on $L_{k}$. Then either $w$ is part of a replacement rule of some $1$-letter word $p$ that appears for the first time on $L_{k-1}$, and we are solving on $L_{k-1}$ the problem of Lemma~\ref{lemma:oneletterword}, or $w$ has a parent $p$ with two letters, $|p|=2$. Then we are recursively solving finding a two-letter word for the first time on $L_{k-1}$. Tracing back the parents of parents, assuming none of them is a single letter, we can see that all of them must be unique $2$-letter words by Lemma~\ref{lemma:firsttimelk}. Hence, we can only go back $n^2-1$ levels before we run out of all the possible ordered pairs of $n$ letters, so one additional step back must bring us to $L_1$. This allows the $2$-letter word $w$ to appear on $L_{n^2+1}$ at the latest as claimed. Note that if we snap back to the problem of $1$-letter words before exhausting all $n^2$ possible pairs, assuming that we have so far seen $m$ letters, we have descended by at most $m^2$ levels and have no more than $n-m$ levels to go before we run out unseen letters and reach $L_1$. Now since $n>m>0$, we have $1<n+m$ and $0<n-m$, so from $n-m<(n+m)(n-m)$, we get $n-m<n^2-m^2$. Thus, $m^2+n-m<n^2$ implies that a $2$-letter word can really appear on $L_{n^2+1}$ at the latest as claimed.
\end{proof}

\begin{example}
In Example~\ref{example:alphabetonedim}, we point out that the word $CC$ never appears. Although one can reach this conclusion intuitively by inspecting the replacement rules, Lemma~\ref{lemma:twoletterwordonedim} tells us that even with the brute-force approach of writing out all the levels, it is enough to check $3^2+1=10$ levels since $L_{10}$ is the latest level on which any $2$-letter word can appear for the first time. Given that $L_{10}$ in this setup has a mere $512$ letters or approximately $7$ lines, it would not be that difficult to verify!
\end{example}

The key to figuring out an upper bound for the latest level on which a word of length greater than $2$ can appear for the first time is determining in how many levels this word can be reduced to a $2$-letter word. Intuitively, while transitioning between levels, the length of the parent on $L_{k-1}$ of the child on $L_k$ will be smaller by a quotient of $b$, so if the parent of $w$ is $p$, we have $|p| \approx \frac{|w|}{b}$. However, as the following example shows, this is really only an approximation.

\begin{example}
In Example~\ref{example:alphabetonedim}, the word $AB$ on $L_3$ has the children $ABA$, $BAC$, and $ABAC$ on $L_4$. The parent of $ABBB$ on $L_4$ is the word $AC$ on $L_3$, while the parent of the overlapping word $CABB$ on $L_4$ is the word $BAC$ on $L_3$.
\end{example}

Therefore, it is useful to establish an upper bound on the length of the parent of any word $w$. The intuition behind focusing on the upper bound is that the longer the parents are along the way of going down levels, the more levels it takes before our initial word reduces to a $2$-letter word, which in turn means that $w$ might be hidden deeper.

\begin{lemma}
\label{lemma:maxparentonedim}
Given an $n$-letter alphabet with $b$-letter replacement rules and a word $w$ on $L_{k}$, an upper bound on the length $|p|$ of a parent $p$ of this word on $L_{k-1}$ is given by
$$|p| \leq \left \lceil \frac{|w| + b-1}{b} \right \rceil.$$
\end{lemma}

\begin{proof}
Fix a length $\ell$ and consider all the possible lengths $|w|$ of a word $w$ for which the maximum length of the parent is $\ell$. The parent $p$ of length $\ell$ on $L_{k-1}$ is replaced by an $\ell b$-letter string. For $w$ on $L_k$ to at least partially cover each $b$-letter section of this $\ell b$-letter string, so that its parent can have length $\ell$, we have
\[(\ell - 2)b+1 < |w|,\]
from which we get
\[(\ell - 1)b < |w| +b -1,\]
and by dividing on both sides, we have
\[\ell - 1 < \frac{|w| +b -1}{b},\]
from which the lemma follows.
\end{proof}

Using this and previous results, we can finally establish an upper bound on the function $F_{1}(b,n,w)$ defined in Section~\ref{section:preliminaries}, which gives the actual latest level on which a specific word $w$ can appear for the first time given an $n$-letter alphabet and $b$-letter replacement rules across all possible configurations of $L_1$. First, we define the function $\mathcal{W}_1$, which is a candidate for the upper bound on $F_{1}$. Note below the subtle difference between the parameters of the functions $F_{1}(b,n,w)$ and $\mathcal{W}_1(b, n, |w|)$. While $F_{1}$ provides the latest level on which a \textit{specific} word $w$ can appear for the first time, i.e., the output may be different for two words of the same length, $\mathcal{W}_1$ provides the same estimate for all the words of equal length regardless of the letters the words are composed of.

\begin{definition}
\label{definition:onedimwordfunction}
Let $\mathcal{W}_1 : \mathbb{N}^3 \mapsto \mathbb{N}$ be a map such that
$$\mathcal{W}_1(b, n, |w|)=
\begin{cases}
  n  & \quad |w| = 1, \\
  \lceil \log_{b}(b|w|-b) \rceil + n^2 & \quad |w| > 1.
\end{cases}
$$
\end{definition}

Next, we prove that $\mathcal{W}_1$ actually is an upper bound of $F_{1}$. We iteratively reduce the word to its largest possible parent, and once we arrive at a $2$-letter word, we already know what to do by Lemma~\ref{lemma:twoletterwordonedim}.

\begin{theorem}
\label{thm:onedim}
Given a one-dimensional grid, an $n$-letter alphabet with $b$-letter replacement rules, and a word $w$, an upper bound on $F_{1}(b,n,w)$ is given by $\mathcal{W}_1(b, n, |w|)$:
\[F_{1}(b,n,w) \leq \mathcal{W}_1(b, n, |w|).\]
\end{theorem}

\begin{proof}
By Lemma~\ref{lemma:oneletterword}, we have $F_{1}(b,n,w) \leq n = \mathcal{W}_1(b, n, 1)$ for $|w|=1$. By Lemma~\ref{lemma:twoletterwordonedim}, we have $F_{1}(b,n,w) \leq n^2+1 = \mathcal{W}_1(b, n, 2)$ for $|w|=2$.

Now we use strong mathematical induction on the length of the word $w$. Assume that there exists a $c \in \mathbb{N}_0$ such that for all words $v$ such that $|v| \leq b^{c} + 1$, the statement holds, and suppose that we have a word $w$ such that
\[b^{c} + 1 < |w| \leq b^{c+1} + 1.\]

Consider the maximum size of the parent $p$ of $w$. By Lemma~\ref{lemma:maxparentonedim}, this size does not exceed 
\[\left \lceil \frac{|w|+b-1}{b} \right \rceil \leq  \left \lceil \frac{b^{c+1}+1 +b-1}{b} \right \rceil = b^c +1.\]

Thus, by the induction assumption, we know that the parent $p$ of $w$ can appear for the first time no later than $\mathcal{W}_1(b, n, b^c+1) = \lceil \log_{b}(b(b^c+1))-b) \rceil + n^2 = c+1+n^2$. Thus, the word $w$ itself appears for the first time no later than $c+2+n^2 = \mathcal{W}_1(b, n, b^{c+1}+1)$ as claimed and $F_{1}(b,n,w) \leq \mathcal{W}_1(b, n, b^{c+1}+1)=\mathcal{W}_1(b, n, |w|)$ for all $|w| > 2$.
\end{proof}

\begin{example}
\label{example:CACABA}
Considering the alphabet and the replacement rules from Example~\ref{example:alphabetonedim}, suppose we know that the word $CACABA$ appears for the first time on $L_k$. We want to find the largest $k$ for which this is possible.

Looking at the replacement rules $AB$, $AC$, and $BB$ for the letters $A$, $B$, and $C$ respectively, we see that if the word $CACABA$ appears later than $L_1$, its parent must be either $BBAA$ or $BBAB$ on $L_{k-1}$. If the parent is $BBAA$, then $L_{k-1}$ is $L_1$ since we cannot have two consecutive letters $A$ on any level but the first. If, however, the parent is $BBAB$, it could have the parent $CA$ on $L_{k-2}$. We continue like this, considering all the possible parents of parents, and stop whenever one of two conditions is met. First, if we reach a word that cannot be a child of any other word, then this word can only be found on $L_1$. Second, if we reach a word on a hypothetical $L_{k-x}$ that has already been seen on a level between $L_{k-x}$ and $L_{k}$, we can ignore it. This approach can be visualized using a graph like in Figure~\ref{fig:CACABA}, where the leaf nodes represent the stopping words in the search.

Hence, we conclude $L_{k-5}$ is the deepest $L_1$ we can reach. If we start with level $L_1$ being $A$, the next few levels are $AB$, $ABAC$, $ABACABBB$, $ABACABBBABACACAC$, followed by $L_6$,
\[ABACABBBABACACACABACABBBABBBABBB,\]
with $CACABA$ on it. Hence, the word $CACABA$ appears on $L_6$ at the latest. This satisfies the constraint given by Theorem~\ref{thm:onedim}, which suggests the upper bound $\lceil \log_{2}(10) \rceil + 9=4+9=13$.
\end{example}

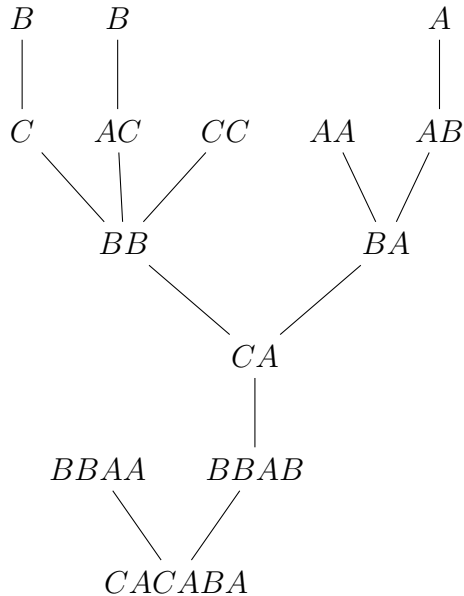
\begin{figure}[ht!]
\begin{center}
\scalebox{1}{
\begin{forest}
for tree={l=1.5cm,s sep=0.5cm,grow'=north}
[$CACABA$
  [$BBAA$]
  [$BBAB$
    [$CA$
      [$BB$
        [$C$
          [$B$]
        ]
        [$AC$
          [$B$]
        ]
        [$CC$]
      ]
      [$BA$
        [$AA$]
        [$AB$
          [$A$]
        ]
      ]
    ]
  ]
]
\end{forest}
}
\end{center}
\caption{Possible lineage for the word $CACABA$}
\label{fig:CACABA}
\end{figure}

In the example above, the latest level on which the word $CACABA$ can appear for the first time is much lower than the upper bound suggested by Theorem~\ref{thm:onedim}. This is a more general observation.

\begin{example}
Using computer-aided verification, one can show that given $2$-letter replacement rules for a $2$-, $3$-, and $4$-letter alphabet, the actual latest levels on which any word can appear for the first time are $4 < 2^2 + 1$, $7 < 3^2 + 1$, and $13 < 4^2 + 1$ respectively\footnote{\href{https://colab.research.google.com/drive/1gmUPLS048wLmHG2SFhGt-l_qMVPvBepr?usp=sharing}{Google Colab Notebook}}.
\end{example}

To conclude, in this section, we have proven an upper bound on the latest level on which a word of any length can appear for the first time in one dimension, with an $n$-letter alphabet and $b$-letter replacement rules, and we can now apply all our findings to the two-dimensional scenario.

\section{Two Dimensions}
\label{section:twodimensions}

In this section, we build on the results from the previous section and find an upper bound on $F(w)$ for vertical, horizontal, and diagonal words in a two-dimensional fractal puzzle. Thus, the implication of the following subsections is that the search in Kisman's puzzle is finite. We start off by defining a useful tool which simplifies the study of diagonal words and also helps collapse the problem of vertical and horizontal words to the already solved problem in one dimension.

\subsection{Bounding Box}

In this subsection, we define the concept of a \textit{bounding box}. The motivation behind this comes from several observations about diagonal words and their parents. We first notice that if a word $w$ is diagonal, its parent $p$ might not be a diagonal word. In Figure~\ref{fig:3letterdiagonal}, the parent of the left word can be represented as three letters in a shape of $L$, $\begin{smallmatrix}*&\\ *&*\end{smallmatrix}$, while the parent of the word on the right are three letters in the shape of inverted $L$, $\begin{smallmatrix}*&*\\ &*\end{smallmatrix}$. Hence, we may need to discuss not only words but also sets of letters.

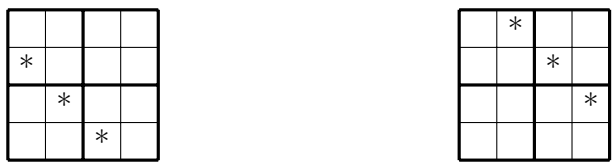
\begin{figure}[ht!]
\begin{center}
\begin{tikzpicture}[scale=.5]
  \begin{scope}
    \draw (0, 0) grid (4, 4);
    \draw[very thick, scale=2] (0, 0) grid (2, 2);

    \setcounter{row}{1}
    \setrow { }{ }  { }{ }
    \setrow {*}{ }  { }{ }
    \setrow { }{*}  { }{ }
    \setrow { }{ }  {*}{ }
  \end{scope}

  \begin{scope}[xshift=12cm]
    \draw (0, 0) grid (4, 4);
    \draw[very thick, scale=2] (0, 0) grid (2, 2);

    \setcounter{row}{1}
    \setrow { }{*}  { }{ }
    \setrow { }{ }  {*}{ }
    \setrow { }{ }  { }{*}
    \setrow { }{ }  { }{ }
  \end{scope}
\end{tikzpicture}
\end{center}
\caption{Possible $3$-letter diagonal words  with $3$-letter parents}
\label{fig:3letterdiagonal}
\end{figure}

Second, the parent of a diagonal word might have just as many letters as the word itself as demonstrated in the example below. Therefore, we require a tool that captures how exactly the parent is smaller.

\begin{example}
\label{example:boxmotivation}
Consider the $3$-letter alphabet $A$, $B$, $C$ with $\left(2 \times 2\right)$-letter replacement rules, where the letters $A$, $B$, and $C$ are replaced by the blocks $\begin{smallmatrix}A&B\\ C&B\end{smallmatrix}$, $\begin{smallmatrix}A&C\\ B&B\end{smallmatrix}$, and $\begin{smallmatrix}B&B\\ C&C\end{smallmatrix}$, respectively. Then the diagonal word $ABAB$ could have the diagonal parent $AB$, and the diagonal word $CBBC$ could have the following parent of width $2$ and height $3$,
\begin{center}
\stretchit{A*}\textcolor{white}{.}\par
\stretchit{CB}\textcolor{white}{.}\par
\stretchit{*B}.\par
\end{center}
\end{example}

As we will see, the bounding box accomplishes both the goal of tracking sets of letters and of capturing the size of a parent in a meaningful way.

\begin{definition}
\label{definition:box}
Let the \textit{bounding box} of a given set of letters in a grid be the smallest rectangle that contains the set.
\end{definition}

\begin{example}
The bounding box of the set of letters
\begin{center}
\stretchit{AB*}\par
\stretchit{**B}\par
\stretchit{*B*}\par
\end{center}
in the two-dimensional grid
\begin{center}
\stretchit{ABAC}\par
\stretchit{CBBB}\par
\stretchit{BBAC}\par
\stretchit{CCBB}\par
\end{center}
is the following $(3 \times 3)$-letter box
\begin{center}
\stretchit{ABA}\textcolor{white}{.}\par
\stretchit{CBB}\textcolor{white}{.}\par
\stretchit{BBA}.\par
\end{center}
\end{example}

Following the above definition, we observe that the width and height of the bounding box of a set of letters $w$ are independent of each other when tracing back the parents of $w$. Thus, we conclude that each dimension of the bounding box of a parent $p$ to a set of letters $w$ behaves according to Lemma~\ref{lemma:maxparentonedim}.

\begin{lemma}
\label{lemma:maxparentbox}
Given a two-dimensional grid, an $n$-letter alphabet with $(b \times b)$-letter replacement rules, and a set of letters $w$ with an $(y \times x)$-letter bounding box on $L_{k}$, each side of the box shrinks on $L_{k-1}$ according to the bound established for words in the one-dimensional scenario.
\end{lemma}

\begin{proof}
Consider the width $x$ of the bounding box. The reasoning for the height $y$ is analogous. We observe that the maximum width of the parent $p$ of the bounding box is constrained by $x$ and the width $b$ of the replacement rules the same way the maximum length of a parent is constrained by the child's length and the length of the replacement rules in the one-dimensional scenario by the proof of Lemma~\ref{lemma:maxparentonedim}.
\end{proof}

\begin{example}
\label{example:maxparentbox}
In Example~\ref{example:boxmotivation}, both examples satisfy the bound $\left \lceil \frac{4+2-1}{2} \right \rceil = 3$ on the length of each of the parent's sides. Notice that for $CBBC$, the width and height of its parent are not equal.
\end{example}

Therefore, thanks to the bounding box, we do not need to worry about whether the parent of $w$ is a diagonal word or a zig-zag formation of letters. Instead, we are concerned about the width and height of the formation. Since each of the sides of the bounding box shrinks with the same speed as a one-dimensional grid, we are left to investigate what happens if the box reduces to a $(2 \times 2)$-letter formation. However, before we do so, we note how the concept of the bounding box collapses the problem of vertical and horizontal words to the already solved problem in one dimension.

\subsection{Vertical and Horizontal}

In this subsection, we directly apply the notion of a bounding box to horizontal and vertical words in the two-dimensional scenario. The following result should not be surprising since intuitively, rows and columns of the two-dimensional grid themselves behave like one-dimensional grids.

\begin{theorem}
\label{thm:twodimvert}
Given a two-dimensional grid and an $n$-letter alphabet with $\left(b \times b\right)$-letter replacement rules, a vertical or horizontal word $w$ is isomorphic in its shrinking to the same word in the one-dimensional scenario, that is
\[F_{2}^{hv}(b,n,w) \leq \mathcal{W}_1(b, n, |w|).\]
\end{theorem}

\begin{proof}
Without loss of generality, consider a horizontal word $w$ since the proof for a vertical word is analogous. The bounding box of $w$ is the word itself, i.e.~it has $1 \times |w|$ letters. We notice that the vertical side of the box is already as small as possible and cannot get larger. By Lemma~\ref{lemma:maxparentbox}, the horizontal side shrinks according to the bound given by Lemma~\ref{lemma:maxparentonedim} until it reaches a length of $2$. Hence, we have reduced the problem to that of a horizontal $2$-letter word. Since a horizontal $2$-letter word can span at most two horizontally neighboring replacement rules and can therefore only have a $1$-letter or horizontal $2$-letter parent, Lemma~\ref{lemma:twoletterwordonedim} can be extended to this situation, and we reach $L_1$ after no more than $n^2$ levels. Hence, the reasoning of Theorem~\ref{thm:onedim} applies, and $\mathcal{W}_1(b, n, |w|)$ is an upper bound on $F_{2}^{hv}(b,n,w)$.
\end{proof}

\begin{example}
Consider the same alphabet and replacement rules as in Example~\ref{example:maxparentbox}. Like in Example~\ref{example:alphabetonedim}, we notice that if $L_1$ does not contain the horizontal or vertical word $AA$, this word never appears horizontally or vertically. Again, one could reach this conclusion intuitively, but Theorem~\ref{thm:twodimvert} tells us that even with the brute-force approach of writing out all the levels, it would be enough to check $3^2+1=10$ levels since $L_{10}$ is the latest level on which any horizontal or vertical $2$-letter word can appear for the first time.
\end{example}

\begin{remark}
Since in a two-dimensional grid, the number of rows/columns grows exponentially with each level, we expect it to be generally faster to find any given word on a horizontal/vertical than to find the same word in a growing one-dimensional grid, i.e., we expect to find the word on an earlier level in two dimensions. However, note that the function $\mathcal{W}_1(b, n, |w|)$ introduced in Definition~\ref{definition:onedimwordfunction} is just as tight an upper bound on $F_{2}^{hv}(b,n,w)$ as it is on $F_{1}(b,n,w)$ since our replacement rules could each have all its rows identical to each other, in which case all rows on any level are identical, and our two-dimensional case is isomorphic to the one-dimensional case.
\end{remark}

Having established an upper bound on $F_{2}^{hv}(b,n,w)$ or the latest level on which a horizontal or vertical word $w$ can appear for the first time, we can now apply the concept of a bounding box to the scenario by which it was motivated, i.e. diagonal words.

\subsection{Diagonal}

Thus, we have shown that horizontal and vertical words in two dimensions differ very little from words in a one-dimensional grid. Once we focus on diagonal words of length greater than $1$, however, the dynamic changes. First of all, the following lemma demonstrates that a diagonal $2$-letter word may be found for the first time at a later level than a horizontal or vertical $2$-letter word. This is due to the fact that a diagonal $2$-letter word can have horizontal, vertical, or diagonal $2$-letter parents.

\begin{lemma}
\label{lemma:twoletterwordtwodimdiag}
Given a two-dimensional grid, an $n$-letter alphabet, and a diagonal $2$-letter word $w$, the latest level on which $w$ can appear for the first time is $L_{2n^2+1}$.
\end{lemma}

\begin{proof}
Suppose $w$ appears for the first time on a diagonal on $L_{k}$. If $w$ is contained in a replacement rule of some $1$-letter word $p$ that appears for the first time on $L_{k-1}$, we are solving on $L_{k-1}$ the problem of Lemma~\ref{lemma:oneletterword}, else $w$ has a parent $p$ with two letters. Then two possible scenarios arise. Either $p$ is vertical or horizontal and $L_{k-1}$ is at most $L_{n^2+1}$ by Theorem~\ref{thm:twodimvert}, so $L_k$ is at most $L_{n^2+2}$. Or $p$ is diagonal, and we are recursively solving finding a diagonal $2$-letter word for the first time on $L_{k-1}$. Tracing back the parents of the parents, assuming none of them is a single letter, we can see all of them must be unique diagonal $2$-letter words. Hence, we can only go back $n^2-1$ levels before we run out of all the possible ordered pairs of $n$ letters on the diagonal, so one additional step back must bring us to $L_1$ or to a vertical or horizontal $2$-letter word. At most $n^2$ levels of diagonal $2$-letter words added to at most $n^2$ levels of vertical or horizontal words allows the diagonal $2$-letter word $w$ to appear on $L_{2n^2+1}$ at the latest as claimed.
\end{proof}

\begin{example}
\label{example:twoletterwordtwodimdiag}
Consider the alphabet and replacement rules from Example~\ref{example:maxparentbox} and the single letter $A$ on $L_1$. Then the first three levels are
\begin{center}
\stretchit{A},\par
\end{center}

\begin{center}
\stretchit{AB}\textcolor{white}{,}\par
\stretchit{CB},\par
\end{center}

\begin{center}
\stretchit{ABAC}\textcolor{white}{.}\par
\stretchit{CBBB}\textcolor{white}{.}\par
\stretchit{BBAC}\textcolor{white}{.}\par
\stretchit{CCBB}.\par
\end{center}
The diagonal word $BB$ appears for the first time on $L_3$. This satisfies the constraint given by Lemma~\ref{lemma:twoletterwordtwodimdiag}, which sets the upper bound on the latest level on which a diagonal $2$-letter word can appear for the first time as $18+1=19$. Notice also that the diagonal word $AA$ can never appear. Although this conclusion can be reached intuitively, one could go through all the $19$ levels to verify the same result manually.
\end{example}

\begin{example}
Consider the alphabet and replacement rules from Example~\ref{example:maxparentbox}. Then, if the word $AC$ does not appear on $L_1$, it cannot appear later on the vertical or the horizontal, but it can appear later on the diagonal.
\end{example}

\begin{remark}
The upper bound given by Lemma~\ref{lemma:twoletterwordtwodimdiag} is very lenient, for it completely disregards the fact that while tracing back all the diagonal $2$-letter words, we will have seen many of all the possible vertical and horizontal $2$-letter words.
\end{remark}

We already showed that a parent of a diagonal word might not be a word. But what shapes are possible? The following lemma explains.

\begin{lemma}
Any ancestor of a diagonal word is contained in two neighboring diagonals.
\end{lemma}

\begin{proof}
Consider letters $A$ and $B$ that have at least one diagonal in between. After applying a $(2 \times 2)$-letter rule, consider letters $X$ and $Y$ that belong to the child blocks of $A$ and $B$. We can see that they have to be separated by at least one diagonal. Continuing further, we cannot get to a diagonal word, which completes the proof.
\end{proof}

\begin{corollary}
\label{corollary:threeletterparent}
If a word ancestor is contained in a $(2 \times 2)$-letter bounding box, it cannot have more than $3$ letters, and if it has exactly $3$ letters, it has to be one of the shapes $\begin{smallmatrix}*&\\ *&*\end{smallmatrix}$ or $\begin{smallmatrix}*& *\\ &*\end{smallmatrix}$.
\end{corollary}

Imagine now a situation like in Figure~\ref{fig:3letterdiagonal} where the parent of a diagonal $3$-letter word is confined in a $(2 \times 2)$-letter box while not being a word itself. As the following lemma demonstrates, such $3$-letter formations have a behavior of their own when we trace their parents.

\begin{lemma}
\label{lemma:twobytwobox}
Given a two-dimensional grid, an $n$-letter alphabet, and a $3$-letter set of letters $w$ inside a $(2 \times 2)$-letter box, the latest level on which $w$ can appear for the first time is $L_{n^3+n^2+1}$.
\end{lemma}

\begin{proof}
For the two possible configurations of $w$, see Figure~\ref{fig:3letterdiagonal}. Suppose $w$ appears for the first time on $L_{k}$. If $w$ is part of a replacement rule of some $1$-letter word $p$ that appears for the first time on $L_{k-1}$, we are solving on $L_{k-1}$ the problem of Lemma~\ref{lemma:oneletterword}. If the parent of $w$ is a $2$-letter word, necessarily horizontal or vertical, the problem reduces to that of Theorem~\ref{thm:twodimvert}. There is, however, one more possible configuration, namely that of a $3$-letter parent $p$ contained within a $\left(2 \times 2\right)$-letter block and shaped as an L with the same rotation as $w$. Parallel to the proof of Lemma~\ref{lemma:twoletterwordonedim}, we can see that this yields a maximum of $n^3-1$ levels before we run out of ordered sets of three letters and reduce to a $1$-letter word or $2$-letter vertical or horizontal word. Afterwards, we have at most $n^2$ levels before we descend to $L_1$. Hence, the latest level on which a $3$-letter set of letters $w$ inside a $(2 \times 2)$-letter box can appear for the first time is $L_{n^3+n^2+1}$.
\end{proof}

We now have all the building blocks in hand and can finally formulate a theorem for an upper bound on the latest level on which any diagonal word can appear for the first time in a two-dimensional grid. We define a function that is an upper bound on the sought $F_{2}^{d}(b,n,w)$.

\begin{definition}
\label{definition:twodimwordfunction}
Let $\mathcal{W}_2 : \mathbb{N}^3 \mapsto \mathbb{N}$ be a map such that
$$\mathcal{W}_2(b, n, |w|)=
\begin{cases}
  n  & |w| = 1, \\
  2n^2 + 1 & |w| = 2 \\
  \lceil \log_{b}(b|w|-b) \rceil +n^2+ n^3 & |w| > 2.
\end{cases}
$$
\end{definition}

\begin{theorem}
\label{thm:twodimdiag}
Given a two-dimensional grid, an $n$-letter alphabet with\linebreak$\left(b \times b\right)$-letter replacement rules, and a diagonal word $w$, an upper bound on $F_{2}^{d}(b,n,w)$ is given by $\mathcal{W}_2(b, n, |w|)$:
\[F_{2}^{d}(b,n,w) \leq \mathcal{W}_2(b, n, |w|).\]
\end{theorem}

\begin{proof}
By Lemmas~\ref{lemma:oneletterword} and \ref{lemma:twoletterwordtwodimdiag}, we have $F_{2}^{d}(b,n,w) \leq n = \mathcal{W}_2(b, n, 1)$ and $F_{2}^{d}(b,n,w) \leq n^2 + 1 = \mathcal{W}_2(b, n, 2)$ as feasible upper bounds on the latest levels on which the diagonal words of lengths $1$ and $2$ can appear.

The key to proving the case for $|w| > 2$ is understanding that unlike in the one-dimensional scenario and unlike in the two-dimensional scenario for vertical and horizontal words, tracing back the parents of parents of a longer diagonal word may never reduce to the problem of a $2$-letter diagonal word. However, we know by Lemma~\ref{lemma:maxparentbox} and Corollary~\ref{corollary:threeletterparent} that we can at least reduce any diagonal word $w$ to a parent contained in a $(2 \times 2)$-letter block on some previous level and that this parent will have no more than $3$ letters.

Now, since both the vertical and horizontal dimensions of the parents of parents of a diagonal word $w$ reduce the same way as the length of a word in the one-dimensional scenario does and since the end goal is $2$ letters in each dimension, we can directly apply the reasoning from the proof of Theorem~\ref{thm:onedim} and conclude that the maximum number of levels required to reach the $\left(2 \times 2\right)$-letter block is given by $\lceil \log_{b}(b|w|-b) \rceil$. However, as observed, the parent contained in this block may have $3$ letters to it. Therefore, after reaching this block, by Lemma~\ref{lemma:twobytwobox}, there will still be up to $n^3+n^2+1$ levels left until definitively reaching $L_1$.
\end{proof}

\section{Solution}
\label{section:solution}

We now have the understanding we need to be able to return to the puzzle solution and connect everything!

Using the language introduced for generalization, the puzzle corresponds to an alphabet of size $26$ and replacement rules of length $b=2$. In Section~\ref{section:introduction}, we last found some words on $L_4$ and noted that writing out every letter on every level would not be an optimal strategy. Indeed, if we spoil that the word $RAUZY$ is hidden furthest on $L_{86}$ and realize that the number of letters in the grid on this level is $11 \times 15 \times 4^{85}$, a number with $54$ digits, we can see that this level alone would require $10^{42}$ terabyte thumb drives just for storage. For reference, at the time of writing, the world's largest single-memory computing system consists of a little over $100$ terabytes.

A more productive way is to solve the puzzle backwards by looking for the possible parents of the missing words, similarly to what we did in Example~\ref{example:CACABA}. As our analysis shows, very fast we get a potential parent contained in a $(2 \times 2)$-letter bounding box and consisting of not more than 3 letters. We also showed that such a parent cannot have more than $26^3 = 17576$ potential $3$-letter parents. This makes searches for all possible parents doable by a program. In reality, the searches are much faster than our bound.

As recalled in the introduction, in some word searches, when one crosses out all the words on the list, one can read a secret message from the remaining letters in the grid. Could this be the case with our puzzle, too? Yes! We can cross out on $L_1$ the words that it contains, as well as those letters whose descendants become parts of the other words from the list located on other levels. Even before we can find all the words from the provided list, the number of remaining letters on level one can be reduced enough for it to be possible to read the secret message: ``SUM EACH WORD'S LEVEL. X MARKS SPOT." This implies that the redundancy of the English language is not enough to solve the puzzle, and we do need to find all the words. 

Once we discover the locations of all the words in the search, we can sum the earliest levels on which these words are found, and this sum will determine the level containing the solution to the puzzle. As the secret message suggests, the solution will be an eventual replacement of the letter $X$, which is not part of any replacement rule and, therefore, only appears once on $L_1$.

Keep in mind that our newly acquired intuition tells us that diagonal words may potentially be hidden later in the puzzle than horizontal and vertical words. Indeed we can see that horizontal grouping $RQ$ expands to a horizontal word $LEVYDRAGON$ on $L_6$. Three more diagonal words $ESCAPE$, $DIMENSION$, and $RAUZY$ appear on $L_{15}$, $L_{17}$, and $L_{86}$ respectively. We see that as we expected diagonal words appear on much later levels in this puzzle than horizontal words.

Finally, by summing the levels of all found words, we determine the level at which to find the solution. And indeed, the $8$-letter answer $HUMPHREY$ appears in the center of level $167$ and is appropriately in the shape of an $X$,

\begin{center}
\stretchit{H**Y}\textcolor{white}{.}\par
\stretchit{*UE*}\textcolor{white}{.}\par
\stretchit{*RM*}\textcolor{white}{.}\par
\stretchit{H**P}.\par
\end{center}

\section{Conclusion}
\label{section:conclusion}

Fractals, a field of study dating back to the 1980s when it was invented by Benoit B. Mandelbrot in \cite{mandelbrot1982fractal}, has since become widely popular, with applications in biology, technology, and beyond. However, the appearance of fractals in a word search puzzle could definitely be considered one of a kind.

An additional parallel to the puzzle studied is the Thue-Morse sequence, which, using our terminology, corresponds to the one-dimensional grid with a $2$-letter alphabet and $2$-letter replacement rules. Suppose the alphabet is $\{0,1\}$, and the rules replace 0 with 01 and 1 with 10. Then, if we start with 0 on $L_1$, we get the Thue-Morse sequence as a limiting sequence. A discussion of the Thue-Morse sequence can be found in \cite{allouche1999ubiquitous}. Hence, the problem we have solved in this paper is not an isolated one but rather one showcasing several areas of mathematics.

We have proven an upper bound on the latest level on which a word $w$ can appear for the first time, given the sizes of the alphabet and the replacement rules, and we have noted that this upper bound is not tight. We have extended our knowledge from the one-dimensional to the two-dimensional scenario and proven two separate upper bounds, one on the latest level on which a vertical or horizontal word can appear for the first time and one on the latest level on which a diagonal word can appear for the first time. Future areas of exploration would be to determine the actual latest levels on which any word can appear in either the one-dimensional or the two-dimensional scenario. Finally, we have used our acquired intuition to provide a solution to Kisman's puzzle.

\section{Acknowledgments}
\label{section:acknowledgements}

The second author thanks Tom Rockicki for encouraging her to look into this topic.

\bibliography{bib}{}

\begin{thebibliography}{1}

\bibitem{allouche1999ubiquitous}
Jean-Paul Allouche and Jeffrey Shallit.
\newblock The ubiquitous {P}rouhet-{T}hue-{M}orse sequence.
\newblock In {\em Sequences and their applications ({S}ingapore, 1998)},
  Springer Ser. Discrete Math. Theor. Comput. Sci., pages 1--16. Springer,
  London, 1999.

\bibitem{kisman2013puzzle}
Derek Kisman.
\newblock In the {D}etails.
\newblock
  \url{https://puzzles.mit.edu/2013/coinheist.com/get_smart/in_the_details/index.html},
  2013.

\bibitem{kisman2013solution}
Derek Kisman.
\newblock In the {D}etails ({S}olution).
\newblock
  \url{https://puzzles.mit.edu/2013/coinheist.com/get_smart/in_the_details/answer/index.html},
  2013.

\bibitem{mandelbrot1982fractal}
Benoit~B. Mandelbrot.
\newblock {\em The fractal geometry of nature.}
\newblock W.~H.~Freeman and Co., San Francisco, Calif., 1982.

\end{thebibliography}
\bibliographystyle{plain}

\end{document}